\documentclass[12pt]{amsart}
\usepackage{hyperref}

\newtheorem{theorem}{Theorem}
\newtheorem{lemma}[theorem]{Lemma}

\newtheorem{corollary}[theorem]{Corollary}
\theoremstyle{definition}
\newtheorem{definition}[theorem]{Definition}

\newcommand{\U}{\Upsilon}
\newcommand{\inv}{^{-1}}
\newcommand{\abs}[1]{\vert #1\vert}
\newcommand{\AF}{\mathcal A}
\newcommand{\BF}{\mathcal B}
\newcommand{\CF}{\mathcal C}
\newcommand{\C}{\mathcal C}

\newcommand{\G}{\Gamma}

\parskip=\smallskipamount

\title[A Geometric Zero-One Law]
{A Geometric Zero-One Law}

\author{Robert H. Gilman}
\address{Department of Mathematical Sciences\\Stevens Institute of Technology\\Hoboken, NJ 07030}
\email{rgilman@stevens.edu}
\author{Yuri Gurevich}
\address{Microsoft Research\\One Microsoft Way\\Redmond, WA 98052}
\email{gurevich@microsoft.com}
\author{Alexei Miasnikov}
\address{Department of Mathematics and Statistics\\McGill University\\Montreal, Quebec H3A 2K6}
\email{alexeim@math.mcgill.ca}


\date{June 1, 2007}

\begin{document}

\begin{abstract} 
Each relational structure $X$ has an associated Gaifman graph, which endows
$X$ with the properties of a graph.  If $x$ is an element of $X$, let $B_n(x)$ be
the ball of radius $n$ around $x$.  Suppose that $X$ is infinite,
connected and of bounded degree.  A first-order sentence $\phi$ in the
language of $X$ is almost surely true (resp.\ a.s.\ false) for finite
substructures of $X$ if for every $x\in X$, the fraction of substructures
of $B_n(x)$ satisfying $\phi$ approaches $1$ (resp.\ $0$) as $n$ approaches
infinity.  Suppose further that, for every finite substructure, $X$ has a
disjoint isomorphic substructure.  Then every $\phi$ is a.s.\ true or
a.s.\ false for finite substructures of $X$.  This is one form of the
geometric zero-one law.  We formulate it also in a form that does not
mention the ambient infinite structure.  In addition, we investigate
various questions related to the geometric zero-one law. 
\end{abstract}

\maketitle

\section{Introduction}

Fix a finite purely relational vocabulary $\U$. From now on structures are $\U$ structures and sentences are first-order $\U$ sentences by default. By substructure we mean the induced substructure corresponding to a subset of elements. All relationships between the elements are inherited, and other relationships are ignored.

According to the well known zero-one law for first-order predicate logic, a first-order sentence $\phi$ is either almost surely true or almost surely false on finite structures~\cite{F}, \cite{GKLT}. In other words if a structure is chosen at random with respect to the uniform distribution on all structures with universe $\{1, 2, \ldots, n\}$, then the probability that $\phi$ is true approaches either $1$ or $0$ as $n$ goes to infinity.

There is another version of the zero-one law in which instead of choosing a structure uniformly at random from the set of structures with universe $\{1, 2,\ldots, n\}$ one chooses an isomorphism class of structures uniformly at random from the set of isomorphism classes of structures with universe of size $n$. This second version is known as the unlabeled zero-one law. The first version, which has received the greater share of attention, is called the labeled zero-one law. It holds for models of parametric axioms, graphs for example, i.e., undirected graphs without loops. For an introduction and surveys see~\cite{C}, \cite[Chapter 3]{EF}, \cite{G}, and~\cite{W}.

There are many extensions of the zero-one law to different logics and
different probability distributions. In this article we consider another
kind of extension. We show in Theorem~\ref{th:main} that under certain
circumstances there is a zero-one law for the finite substructures of a
fixed infinite structure; Theorem~\ref{th:main-version2} gives a variation on this theme which does not refer to the ambient infinite
structure. Theorem~\ref{th:more} shows that our results can yield zero-one laws for classes of structures to which neither the labeled nor unlabeled law applies. 

Let $X$ be a fixed infinite structure. If $X$ were finite, a natural way
to compute the probability that a finite substructure satisfied a
sentence $\phi$ would be to divide the number of substructures of $X$
satisfying $\phi$ by the total number of substructures of $X$. As $X$ is
infinite, this simple approach does not work; but there is a
straightforward extension which does. To explain it we need a few definitions.

Recall that the Gaifman graph~\cite{Ga} of $X$ has the elements of $X$ as its vertices and an undirected edge between any two distinct vertices, $x, y$, for which there is a relation $R\in \U$ and elements $z_1,\ldots z_\ell$ in $X$ such that $R(z_1,\ldots,z_\ell)$ is true in $X$ and $x,y \in \{z_1, \ldots z_\ell\}$. Denote the Gaifman graph of $X$ by $[X]$. 

If $X$ is a graph, we may identify $X$ with $[X]$. In any case we extend some standard graph-theoretic terminology from $[X]$ to $X$. The distance, $d(x,y)$, between $x,y\in X$ is the length of the shortest path from $x$ to $y$ in $[X]$ or $\infty$ if there is no such path. For any $Y\subseteq X$, $d(x, Y)$ is the minimum distance from $x$ to a point in $Y$, and $B_n(Y)$ is the substructure of $X$ supported by the elements a distance at most $n$ from $Y$. $B_n(x)$ is an abbreviation of $B_n(\{x\})$. The ambient structure $X$ to which $B_n(Y)$ and $B_n(x)$ refer will be clear from the context.  

Two substructures of $X$ are said to be disjoint if their intersection is empty and there are no edges between them in $[X]$. The disjoint union of structures is defined in the obvious way. Substructures corresponding to the connected components of $[X]$ are called components of $X$, and substructures which are unions of components are called closed. A structure with just one component is said to be connected. If all vertices of $[X]$ have finite degree, $X$ is locally finite; and if the vertex degrees are uniformly bounded, $X$ has bounded degree.

\begin{definition}\label{de:almosttrue} 
Suppose $X$ is an infinite, connected, locally finite structure. A sentence
is almost surely true for finite substructures of $X$ if for every $x\in
X$ the fraction of substructures of $B_n(x)$ for which the sentence is
true approaches $1$ as $n$ approaches infinity.
\end{definition}

The balls $B_n(x)$ mentioned in Definition~\ref{de:almosttrue} are finite because $X$ is locally finite.

\begin{definition}\label{de:duplicate substructure} 
A structure $X$ has the duplicate substructure property if for every finite substructure there is a disjoint isomorphic substructure.
\end{definition}

\begin{theorem}\label{th:main} Let $X$ be an infinite connected structure of bounded degree and possessing the duplicate substructure property. Then any sentence is either almost surely true or almost surely false for finite substructures of $X$.
\end{theorem}

We may think of the structure $X$ from Theorem~\ref{th:main} as inducing a
zero-one law on the class $\C(X)$ of its finite
substructures. $\C(X)$ is closed under substructures and
disjoint union.  Also, $\C(X)$ is pseudo-connected in the following sense.

\begin{definition} 
A class $\CF$ of finite structures is \emph{pseudo-connected} if, for
every $Y\in\CF$ there is an embedding of $Y$ into a connected member of
$\CF$.
\end{definition}

\begin{theorem}\label{th:main-version2}
Let $\CF$ be a pseudo-connected class of finite structures of bounded
degree closed under substructures and disjoint unions, and let $S$ be the
disjoint union of all members of $\CF$.  We have:
\begin{enumerate}
\item 
There is an infinite structure $X$, called an \emph{ambient structure} for
$\CF$, such that $X$ satisfies the conditions of Theorem~\ref{th:main} and
$\CF$ is the collection of (isomorphic copies of) substructures of $X$.
\item
An arbitrary first-order sentence $\phi$ is almost surely true for $\CF$
if and only if it holds in $X$.
\end{enumerate}
\end{theorem}

Thus such a class $\CF$ always has an ambient structure, and
  different ambient structures induce the same zero-one law on $\CF$.

The proof of Theorem~\ref{th:main} proceeds along a well known path. We show that certain axioms are almost surely true for finite substructures of $X$ and that the theory with those axioms is complete. Section~\ref{se:theory} contains the proof of Theorem~\ref{th:main} and a discussion of the almost sure theory. In Sections~\ref{se:random} and~\ref{se:trees} we show that random substructures of $X$ are elementarily equivalent but not necessarily isomorphic.  This result may have application to the theory of percolation. See~\cite{BS1, BS2}. We thank Andreas Blass for useful discussions related to Section~\ref{se:trees}.
 
Now we present some examples. It is straightforward to check that Theorem~\ref{th:main} applies to the following structures.

\begin{enumerate}
\item The Cayley diagram of a finitely generated infinite group. Here $\U$ consists of one binary relation for each generator.
\item\label{ex:underlying graph} An infinite connected vertex-transitive graph of finite degree. For example the graph obtained from a Cayley diagram of the type just mentioned by removing all loops and combining all edges between any two distinct vertices joined by an edge into a single undirected edge. See~\cite{JS} for non-Cayley examples.
\item The Cayley diagram of a free finitely generated monoid.
\item The full binary tree; i.e., the tree with one vertex of degree two and all others of degree three. More generally the full $k$-ary tree for $k \ge 1$.
\item An infinite connected locally finite and finite dimensional simplicial complex whose automorphism group is transitive on zero-simplices. There is one $n+1$-ary relation for each dimension $n$.
\end{enumerate}

We conclude this section with an example of a class of structures which satisfies the geometric zero-one law, but for which neither the labeled nor unlabeled law holds. For this purpose a unary forest is defined to be a directed graph such that each vertex has at most one incoming edge and at most one outgoing edge.

A unary tree is a connected unary forest; that is, a directed graph consisting of a single finite or infinite directed path. $\CF$ is the class of finite unary forests with edges labeled by $0$ and $1$; $\U$ consists of two binary relations, one for each edge label. $\CF$ is closed under isomorphism, disjoint union, and restriction to components.

\begin{theorem}\label{th:more}
$\CF$, the class of finite unary forests with edges labeled by $0$ and $1$, obeys the geometric zero-one law but does not obey either the labeled or unlabeled law.
\end{theorem}

\begin{proof} Pick an infinite labeled unary tree, $X$, such that all finite sequences of $0$'s and $1$'s appear as the labels of subtrees of $X$; observe that  $X$ satisfies the hypotheses of Theorem~\ref{th:main}. Thus $\CF$ obeys the geometric zero-one law.
 
To show that $\CF$ does not satisfy the labeled or unlabeled law, we apply \cite[Theorem~5.9]{C1}. Let $\AF_n$ be the set of structures in $\CF$ with universe $\{1, 2,\ldots, n\}$, and $\BF_n$ a set of representatives for the isomorphism classes of structures in $\AF_n$. The cardinalities of $\AF_n$ and $\BF_n$ are denoted $a_n$ and $b_n$ respectively. It follows immediately from \cite[Theorem~5.9]{C1} that if $\sum_{n=1}^\infty \frac{a_n}{n!}t^n$ has finite positive radius of convergence, then $\CF$ does not obey the labeled zero-one law. Likewise if $\sum_{n=1}^\infty b_nt^n$ has radius of convergence strictly between $0$ and $1$, then $\CF$ does not obey the unlabeled zero-one law.

Consider a single unary tree with $n$ vertices. The $2^{n-1}$ different ways of labeling the edges of this tree yield pairwise non-isomorphic labeled trees; and for each labeled tree, the $n!$ different ways of labeling the vertices yield different structures on $\{1,2,\ldots, n\}$. Thus $2^{n-1} \le b_n$ and $2^{n-1}n! \le a_n$. On the other hand each unary forest of size $n$ is isomorphic to a structure obtained by labeling the edges of a unary tree of size $n$ with letters from the alphabet $\{0,1,2\}$ and then deleting all edges with label $2$. It follows that $2^{n-1} \le b_n \le 3^{n-1}$ and $2^{n-1}n! \le a_n \le 3^{n-1}n!$. By the results mentioned above neither the labeled nor unlabeled zero-one law holds for $\CF$.
\end{proof}

\section{A Sufficient Condition for Elementary Equivalence}

The main result of this section is that two structures which satisfy the following condition are elementarily equivalent.

\begin{definition}\label{de:disjoint ball} 
Two structures satisfy the disjoint ball extension condition if whenever either structure contains a ball $B_n(x)$ disjoint from a finite substructure $F$, and the other structure has a substructure $F'$ isomorphic to $F$, then the other structure also contains $B_n(y)$ disjoint from $F'$ isomorphic to $B_n(x)$ by an isomorphism matching $x$ to $y$.
\end{definition}

\begin{lemma}\label{le:isomorphisms}
Let $X$ and $X'$ be structures and $Y$ a substructure of $X$. If $\alpha$ is an isomorphism of $B_n(Y)$ to a substructure of $Y'$, the following conditions hold.
\begin{enumerate}
 \item If $x_1\in B_{n-1}(Y)$ and $x_2\in B_n(Y)$ are joined by an edge in $[X]$, then $\alpha(x_1)$ and $\alpha(x_2)$ are joined by an edge in $[X']$.
 \item For any $x\in B_n(Y)$, $d(x, Y)\ge d(\alpha(x), \alpha(Y))$.
 \item\label{it:inclusion} $\alpha(B_n(Y))\subseteq B_n(\alpha(Y))$.
 \item If $\alpha$ $maps(B_n(Y))$ onto $B_n(\alpha(Y))$, then for any $x\in B_n(Y)$, $d(x, Y) = d(\alpha(x), \alpha(Y))$.
\end{enumerate}
\end{lemma}

\begin{proof}
If $x_1, x_2$ are as above, then $R(t_1,\ldots,t_k)$ is true for some relation $R\in \U$ and elements $t_1,\ldots, t_k \in X$ with $x_1, x_2\in \{t_1,\ldots, t_k\}$. It follows that $d(x_1, t_i)\le 1$ for all $i$, which implies $\{t_1,\ldots, t_k\}\subseteq B_n(Y)$. As $\alpha$ is an isomorphism, $R(\alpha(t_1),\ldots, \alpha(t_k))$ holds in $X'$. Thus the first part is proved. The first part implies the next two, and the last one holds by symmetry. 
\end{proof}

\begin{lemma}\label{le:isometries} 
Let $X$ and $X'$ be structures. Suppose that for some $n\ge 1$ and substructures $Y\subseteq X$, $Y'\subseteq X'$ there is an isomorphism $\alpha:B_n(Y)\to B_n(Y')$ with $\alpha(Y)=Y'$. Then for any substructure $Z$ of $X$ with $B_m(Z)\subseteq B_{n-1}(Y)$, $\alpha$ maps $B_m(Z)$ isomorphically to $B_m(\alpha(Z))$.
\end{lemma}

\begin{proof}
First suppose that $B_m(\alpha(Z))\subseteq B_{n-1}(Y')$. Lemma~\ref{le:isomorphisms}(\ref{it:inclusion}) applied to $\alpha$ and $\alpha\inv$ yields $\alpha(B_m(Z)) \subseteq B_m(\alpha(Z))$ and $\alpha\inv(B_m(\alpha(Z))) \subseteq B_m(Z)$. It follows immediately that $\alpha$ maps $B_m(Z)$ isomorphically to $B_m(\alpha(Z))$ as desired.
 
Thus it suffices to show that $B_m(\alpha(Z))\subseteq B_{n-1}(Y')$. Assume not. As $\alpha(Z) \subseteq B_{n-1}(Y')$, there must be an element $\alpha(x) \in B_n(Y') - B_{n-1}(Y')$ with $d(\alpha(x), \alpha(Z))=k\le m$. Consequently there is a path in $[X']$ from some $\alpha(z)\in \alpha(Z)$ to $\alpha(x)$ of length at most $m$ and with all vertices of the path in $B_m(\alpha(Z))$. Without loss of generality assume that $\alpha(x)$ is the first point on that path not in $B_{n-1}(Y')$. But then Lemma~\ref{le:isomorphisms} implies $x\in B_m(Z) - B_{n-1}(Y)$ contrary to hypothesis.
\end{proof}

\begin{theorem}\label{th:equivalence} If two locally finite structures satisfy the disjoint ball extension condition, then they are elementarily equivalent. 
\end{theorem}

\begin{proof} 
Let $X$ and $X'$ be the two structures. We show that for each $n$ the duplicator can win the $n$-step Ehrenfeucht game by constructing isomorphisms $\alpha_i$ from a substructure $F_i\subseteq X$ to a substructure $F_i'\subseteq X'$, where $F_i$ and $F_i'$ consist of the elements chosen by the spoiler and the duplicator in the first $i$ steps. Each $\alpha_i$ will be the restriction of an isomorphism, also called $\alpha_i$, from $B_{5^{n-i}}(F_i)$ to $B_{5^{n-i}}(F_i')$.

We argue by induction on $i$. Suppose $i=1$. By symmetry we may suppose that the spoiler picks $ x\in X$. By hypothesis there is an isomorphism $\alpha_1:B_{5^{n-1}}(x)\to B_{5^{n-1}}(x') \subseteq X'$ with $x'=\alpha_1(x)$. The duplicator chooses $x'$.

Assume $\alpha_i:B_{5^{n-i}}(F_i) \to B_{5^{n-i}}(F_i')$ is an isomorphism for some $i<n$. Again by symmetry the spoiler picks $x\in X$. We have $F_{i+1}=F_i\cup\{x\}$. If $B_{5^{n-i-1}}(x)\subseteq B_{5^{n-i}-1}(F_i)$, then we take $\alpha_{i+1}$ to be the restriction of $\alpha_i$ to $B_{5^{n-i-1}}(F_{i+1})$ and set  $x' = \alpha_i(x)$, $F_{i+1}' = F_i'\cup\{x'\}$. By Lemma~\ref{le:isometries}, $\alpha_{i+1}$ maps $B_{5^{n-i-1}}(F_{i+1})$ onto $B_{5^{n-i-1}}(F_{i+1}')$.

Otherwise $B_{5^{n-i-1}}(x)$ is not a subset of $B_{5^{n-i}-1}(F_i)$. Some $y\in B_{5^{n-i-1}}(x)$ must be a distance at least $5^{n-i}$ from $F_i$. Thus the distance of every vertex $z\in B_{5^{n-i-1}}(x)$ from $F_i$ is at least $5^{n-i} - d(y, z) \ge 5^{n-i} - 2(5^{n-i-1}) \ge 3(5^{n-i-1})$ from $F_i$. It follows that $B_{5^{n-i-1}}(x)$ and $B_{5^{n-i-1}}(F_i)$ are a distance at least $3(5^{n-i-1}) - 5^{n-i-1}\ge 2(5^{n-i-1})\ge 2(5^0)=2$. Thus $B_{5^{n-i-1}}(x)$ and $B_{5^{n-i-1}}(F_i)$ are disjoint.

By hypothesis there is an isomorphism $\beta:B_{5^{n-i-1}}(x)\to B_{5^{n-i-1}}(x')$ with $\beta(x)=x'$ and $B_{5^{n-i-1}}(x')$ disjoint from $\alpha_i(B_{5^{n-i-1}}(F_i))$. Combining the restriction of $\alpha_i$ to $B_{5^{n-i-1}}(F_i)$ with $\beta$, we obtain $\alpha_{i+1}$.
\end{proof}

\section{The Almost Sure Theory}\label{se:theory}

Fix an infinite connected structure $X$ of bounded degree satisfying the duplicate substructure property. Let $\mathcal C$ be the collection of all structures isomorphic to finite substructures of $X$. By construction $\mathcal C$ is closed under passage to substructures. By the duplicate substructure property of $X$, $\mathcal C$ is closed under disjoint union. 

Let $\mathcal A$ be a set of representatives for the isomorphism classes of all finite structures, and define sentences $\sigma_F$, $F\in \mathcal A$, as follows. For $F\in \mathcal A \cap \mathcal C$, $\sigma_F$ says that there there is a closed substructure isomorphic to $F$; for $F\in \mathcal A - \mathcal C$, $\sigma_F$ says that there is no substructure isomorphic to $F$. Define $T$ to be the theory with axioms $\{\sigma_F\}$. 

Observe that the disjoint union of $\{F \mid F\in \mathcal A \cap \mathcal C \}$ is a model of $T$. 

\begin{lemma}\label{le:many} The following conditions hold for any model $Y$ of $T$.
\begin{enumerate}
\item\label{it:many1} Every finite substructure of $Y$ is isomorphic to a closed substructure;
\item\label{it:many2} For any two finite substructures, there is a finite substructure isomorphic to their disjoint union. 
\item\label{it:many3} The union of all finite closed substructures of $Y$ is a model of $T$ and consists of infinitely many disjoint copies of each finite substructure of $X$.
\end{enumerate}
\end{lemma}

\begin{proof} Item~(\ref{it:many1}) and the first part of~(\ref{it:many3}) hold by construction of $T$.
For~(\ref{it:many2}) observe that as $\mathcal C$ is closed under disjoint union, for any $F_1, F_2\in \mathcal A\cap \mathcal C$ there is an $F_3 \in \mathcal A\cap \mathcal C$ isomorphic to the disjoint union of $F_1$ and $F_2$. Finally the last part of~(\ref{it:many3}) follows from~(\ref{it:many1}) and~(\ref{it:many2}).
\end{proof}

\begin{lemma}\label{le:complete} $T$ is complete.
\end{lemma}

\begin{proof} 
It suffices to show that any two models of $T$ are elementarily
equivalent. Up to isomorphism the finite substructures of any model of $T$
are the same as those of $X$. Thus models of $T$ have bounded degree. By
Theorem~\ref{th:equivalence} it suffices to show that any two models $Y$, $Y'$ of
$T$ satisfy the disjoint ball extension condition.

Suppose that $F$ is a finite substructure of $Y$ and $B_n(y)\subseteq Y$ is
disjoint from $F$, and $F$ is isomorphic to $F'\subseteq Y'$. $B_n(y)$ is a
finite substructure of $Y$ and hence isomorphic to a finite closed
substructure $Z'\subseteq Y'$. By Lemma~\ref{le:many} we may assume $Z'$ is
disjoint from $F'$. Let $y'$ be the image of $y$ under this isomorphism
mapping $B_n(y)$ to $Z'$. By Lemma~\ref{le:isomorphisms}, $Z'\subseteq
B_n(y')$. As $Z'$ is closed, it follows that $Z' = B_n(y')$.
\end{proof}

\begin{lemma}\label{le:sigma-i-generic}
  Each axiom $\sigma_F$ is almost surely true for finite substructures of $X$.
\end{lemma}

\begin{proof} 
If $\sigma_F$ says there is no substructure isomorphic to $F$, then $F$ is not isomorphic to any substructure of $X$. Hence $\sigma_F$ holds for all substructures of every ball in $X$.  In the remaining case $\sigma_F$ says that there is a closed substructure isomorphic to $F$. It follows that $F$ is isomorphic to a substructure $F_1$ of $X$.

Choose $F_1$ such that $G_1=B_1(F_1)$ has maximum possible size, $k$. This is possible because the vertex degree of $[X]$ is bounded. $G_1$ has $2^k$ subsets, one of which supports $F_1$. Further our choice of $F_1$ guarantees that if $G'$ is any substructure isomorphic to $G_1$, then $G'=B_1(F')$ for some substructure $F'$ isomorphic to $F$. By Lemma~\ref{le:many} there are denumerably many substructures $G_2, G_3, \ldots$ isomorphic to $G_1$ and disjoint from $G_1$ and each other.  Each $G_i$ is $B_1(F_i)$ for a substructure $F_i$ of $G_i$ isomorphic to $F$.

Consider balls $B_n(x)$ for some $x$. It follows from the connectedness of $X$ that for any $m$, $B=B_n(x)$ will contain at least $m$ of the $G_i$'s if $n$ is large enough. For each $G_i\subseteq B$, the fraction of substructures of $B$ whose restriction to that $G_i$ is not $F_i$ is at most $1-2^{-k}$.  Thus the fraction whose restriction to some $G_i$ in $B_n(x)$ equals $F_i$ is at least $1- (1-2^{-k})^m$, which is arbitrarily small when $m$ is large enough and hence when $n$ is large enough. Further when the restriction of a substructure of $B$ to $G_i$ is $F_i$, then because the substructure does not contain any points of $B_1(F_i)-F_i$, $F_i$ is closed in the substructure.
\end{proof}

Now we complete the proof of Theorem~\ref{th:main}. Let $\sigma$ be
an arbitrary first-order sentence in the language of graph theory.
Since $T$ is complete it follows that either $\sigma$ or $ \neg
\sigma$ is derivable from a finite set of axioms of $T$. Clearly the
conjunction of this finite set of almost surely true sentences is almost
true for finite substructures of $X$. It follows that $\sigma$ or
$\neg \sigma$, whichever one is derivable from $T$, is almost surely true
for finite substructures of $X$. The proof of Theorem~\ref{th:main} is
complete.

\section{Decidability}\label{se:decidability}

In this and subsequent sections we develop our theme further. From now on $X$ is any structure satisfying the hypotheses of Theorem~\ref{th:main} and $T$ is the almost sure theory for finite substructures of $X$. 

\begin{definition}\label{de:locally computable} 
$X$ is locally computable if for every natural number $n$ one can effectively find a set of representatives of the isomorphism classes of balls of radius $n$. 
\end{definition}

Notice that by hypothesis $X$ is of bounded degree. Thus for any $n$ there are up to isomorphism only a finite number of balls of radius $n$.

\begin{lemma}\label{le:decidable}
 $T$ is decidable if and only if $X$ is locally computable.
\end{lemma}

\begin{proof}
Assume $X$ is locally computable. To prove that $T$ is decidable, it suffices to show that the axioms for $T$ are computable. Indeed if the axioms are computable, then $T$ is recursively enumerable; and because $T$ is complete, enumeration of $T$ produces either $\sigma$ or $\neg\sigma$ for every sentence $\sigma$. Thus $T$ is decidable.

The axioms of $T$ are computable if we can decide for any finite structure $F$ whether or not $F$ is isomorphic to a substructure of $X$. If $[F]$ is connected, then any isomorphic substructure $F_1$ of $X$ must lie in some ball of radius at most equal to the size of $F$. By hypothesis we can examine the finitely many representatives of the isomorphism classes of these balls to check if $F$ is isomorphic to a substructure of $X$.

In general we can check in the same way if the substructures $C$ of $F$ corresponding to the connected components of $[F]$ are isomorphic to substructures of $X$. If some $C$ fails the test, then $F$ cannot be a substructure of $X$. If they all pass, then by the duplicate substructure property we can embed them into $X$ is such a way that elements of distinct $C$'s are a distance at least $2$ from each other. It follows that the union of the $C$'s is isomorphic to $F$.

To prove the converse suppose that $T$ is decidable. For any finite structure $F$ one can write down a formula which says that there is an element $u$ for which the ball of radius $n$ around $u$ is isomorphic to $F$. Hence one can decide whether or not $F$ is isomorphic to a ball of radius $n$ in $X$. As $X$ has bounded degree, only finitely many $F$'s have to be checked in order to generate a complete list of isomorphism types of balls of radius $n$ in $X$.
\end{proof}

\begin{corollary} If $X$ is the Cayley diagram of a finitely generated group $G$, then $T$ is decidable if and only if $X$ has solvable word problem
\end{corollary}

\begin{proof} Reall that there is one binary predicate for each generator of $G$. If the word problem is decidable, one can construct the ball of radius $n$ around the identity. Since all balls of radius $n$ are isomorphic, $X$ is locally computable. Conversely if $X$ is locally computable, $T$ is decidable by Lemma~\ref{le:decidable}. For any word $w$ in the generators of $G$, the binary relation $R_w(x,y)$ which holds when there is a path with label $w$ from $x$ to $y$ in $X$ is definable. Thus we can decide if $\exists x\, R_w(x,x)$ is true, i.e., if $w$ defines the identity in $G$.
\end{proof}

\section{Random Substructures}\label{se:random}

For a fixed $p$, $0 < p < 1$, we may imagine generating a random substructure of $X$ by deleting each element of $X$ with probability $1-p$. The random substructure is the one supported by all the remaining elements. We will show that almost all random substructures are elementarily equivalent but not necessarily isomorphic.

A more precise definition of random substructures of $X$ is obtained by first defining a measure on cones. For each pair, $S,T$, of disjoint finite subsets of elements of $X$, the corresponding cone consists of all subsets of elements which include $S$ and avoid $T$. The measure of this cone is defined to be $p^{\abs{S}}q^{\abs{T}}$, where $\abs S$ and $\abs{T}$ are the cardinalities of $S$ and $T$ respectively, and $q=1-p$. By a well known theorem of Kolmogorov the measure on cones extends uniquely to a probability measure, $\mu$, on the $\sigma$--algebra generated by the cones. 

\begin{lemma}
\label{le:union-delta}
Let $F$ be a finite substructure of $X$. With probability $1$ a random substructure of $X$ contains a closed substructure isomorphic to $F$.
\end{lemma}

\begin{proof} The proof is just a modification of the proof of Lemma~\ref{le:sigma-i-generic}. Fix $F$, and pick a substructure $F_1$ of $X$ which is isomorphic to $F$ and for which $B_1(F_1)$ is maximal. By the duplicate substructure property $X$ has denumerably many pairwise disjoint and isomorphic substructures $H_1=B_1(F_1), H_2, H_3, \ldots$. For any $i$ there is an isomorphism $\alpha_i:H_1\to H_i$ carrying $F_1$ to $F_i=\alpha(F_1)$. By Lemma~\ref{le:isomorphisms} $H_i\subseteq B_1(F_i)$. By maximality of $B_1(F_1)$ we have $H_i = B_1(F_i)$. 

Let $Y$ be a random substructure of $X$. If $Y\cup B_1(F_i) = F_i$, then $Y$ contains $F_i$ as a closed substructure. By disjointness the denumerably many events $Y\cap B_1(F_i) \neq F_i$ are independent. As each of these event has the same probability, and that probability is less than $1$, we conclude that the probability of a random graph containing at least one of the $F_i$'s as a closed substructure is $1$. 
\end{proof}

Now define $X^*$ to be the structure consisting of the disjoint union of a denumerable number of copies of each finite substructure of $X$. It is clear that $X^*$ is a model $T$.

\begin{lemma}\label{le:submodel} With probability $1$ a random substructure of $X$ contains a closed substructure isomorphic to $X^*$. 
\end{lemma}

\begin{proof}
The duplicate substructure property and Lemma~\ref{le:union-delta} together guarantee that the set of substructures with the desired property is the intersection of a countable number of sets of measure $1$. 
\end{proof}

\begin{theorem}\label{th:random} With probability $1$ a random substructure of $X$ is a model of $T$. In particular, almost all random substructures of $X$ are elementarily equivalent.
\end{theorem}

\begin{proof} By Lemma~\ref{le:submodel} it suffices to show that if a substructure $X_0$ of $X$ contains a union of connected components isomorphic to $X^*$, then $X_0$ is elementarily equivalent to $X^*$. The argument used in the proof of Lemma~\ref{le:complete} applies.
\end{proof}

\section{Random Subgraphs of Trees}\label{se:trees}
In this section we obtain more precise results for random subtrees of trees.

Let $\G_k$, $k\ge 1$, be the full $k$-ary tree, that is, the tree with one vertex, the root, of degree $k$ and all others of degree $k+1$. A descending path in $\G_k$ is one which starts at any vertex and continues away from the root.

As we noted earlier, Theorem~\ref{th:main} applies to $\G_k$. We maintain the following notation: $p$ is a number strictly between $0$ and $1$, $q=1-p$, and $\mu$ is the corresponding measure on subgraphs of $\G_k$.

Let $p_n$ be the probability that a random subgraph admits no descending path of length $n$ starting at a fixed vertex $v$. A moments thought shows that $p_0=q$, and $p_{n+1}=q + pp_n^k$. In particular $p_n$ is independent of the choice of $v$. The probability that a random subtree contains an infinite descending path starting at a particular vertex $v$ is $1 - \lim_{n \to \infty}p_n$.

\begin{lemma}\label{th:infinite path} The probability that a random subtree contains an infinite descending path starting at a particular vertex $v$ is $0$ if $p \le 1/k$ and strictly between $0$ and $1$ otherwise.
\end{lemma}

\begin{proof} Define $f(x)=q + px^k$. Observe that $f(0)= q=p_0$, $f(f(0))=p_1$, etc. Further $f$ maps the unit interval to itself and is strictly increasing on that interval. Thus $p_0, p_1, p_2, \ldots$ is an increasing bounded sequence which converges to a fixed point of $f$. When $k=1$, $f$ is linear with a single fixed point (on the unit interval) at $x=1$. Otherwise $f$ is concave up and has a single fixed point at $x=1$ if $p\le 1/k$ and two fixed points if $p > 1/k$. Let $x_0$ be the least fixed point of $f$ on the unit interval. $0 \le x_0$ implies that every point in the forward orbit of $0$ under $f$ is no greater than $x_0$. Thus $p_0, p_1, p_2, \ldots$ converges to $x_0$. As $0 < q \le x_0$, we are done.
\end{proof}

We observe that the statement that there is an infinite descending path starting at the root of a full $k$-ary tree can be formulated in monadic second-order logic, in fact in existential monadic second-order logic. Thus we have evidence that Theorem~\ref{th:random} does not extend to this more powerful logic.

\end{document}